\documentclass{amsart}
\usepackage[margin=1in]{geometry}
\usepackage{amsthm}
\usepackage{amssymb}
\usepackage[non-compressed-cites]{amsrefs}
\usepackage{enumitem}
\usepackage{mathtools}
\usepackage{xcolor}
\usepackage{hyperref}

\newcommand{\EQ}[1]{\begin{equation}\begin{split} #1 \end{split}\end{equation}}
\newcommand{\R}{\mathbb{R}}
\newcommand{\C}{\mathbb{C}}
\newcommand{\Sph}{\mathbb{S}}

\newcommand{\J}{\mathcal{J}}
\newcommand{\del}{\partial}

\newcommand\inner[2]{\left\langle#1,#2\right\rangle}

\newtheorem{theorem}{Theorem}[section]
\newtheorem{corollary}[theorem]{Corollary}
\newtheorem{lemma}[theorem]{Lemma}

\newtheorem{proposition}[theorem]{Proposition}
\newtheorem{remark}[theorem]{Remark}
\numberwithin{equation}{section}
 
\hypersetup{
    colorlinks,
    linkcolor={red!60!black},
    citecolor={blue!80!black},
}

\mathtoolsset{showonlyrefs}

\title[Concavity in Locally Symmetric Spaces with nonnegative curvature]{Concavity for elliptic and parabolic equations in locally symmetric spaces with nonnegative curvature}
\author[Shrey Aryan]{Shrey Aryan}
\author[Michael B. Law]{Michael B. Law}
\address{MIT, Department of Mathematics, 77 Massachusetts Avenue, Cambridge, MA 02139, USA.}
\email{shrey183@mit.edu, mikelaw@mit.edu}

\begin{document}

\begin{abstract}
We establish a concavity principle for solutions to elliptic and parabolic equations on locally symmetric spaces with nonnegative sectional curvature, extending the results of Langford and Scheuer in \cite{langford2021concavity}. To the best of our knowledge, this is the first general concavity principle established on spaces with non-constant sectional curvature. 
\end{abstract}
\maketitle
\section{Introduction} \label{sec:intro}
Concavity principles form a key theme in the study of elliptic and parabolic PDEs. They explain the extent to which the convexity of the domain influences the convexity of solutions to a PDE.
Numerous concavity principles are known on Euclidean space as well as spaces of constant curvature; see the discussion in \S\ref{subsec:background}. 

In this paper, we develop a concavity principle for certain nonlinear elliptic and parabolic PDEs on locally symmetric spaces with nonnegative sectional curvature. Here, locally symmetric is in the Riemannian sense, i.e. the curvature tensor is parallel. Examples include Euclidean spaces $\R^n$, spheres $\Sph^n$, projective spaces and Grassmannians over $\R$, $\C$ or $\mathbb{H}$, as well as any products of these spaces. In particular, the sectional curvature need not be constant. Extending concavity principles to such settings has been a long-standing problem. Our work builds on \cite{langford2021concavity} where the same results were established on $\Sph^n$.
More precisely, we show
\begin{theorem}[Elliptic concavity principle] \label{thm:elliptic}
Let $\mathcal{M}$ be a locally symmetric space of dimension $n \geq 1$ with sectional curvatures lying in the interval $[0,A]$. Let $\Omega \subset \mathcal{M}$ be a domain with geodesically convex closure and diameter strictly less than $\frac{\pi}{\sqrt{A}}$. Let $\Gamma \subset \mathbb{R}^n$ be a symmetric, open and convex cone containing $\Gamma_{+}=\left\{\kappa \in \mathbb{R}^n: \kappa_i>0 \;\; \forall 1 \leq i \leq n\right\}$, and let $u \in C^2(\Omega) \cap C^1(\bar{\Omega})$ be a solution to
\begin{align}\label{eq:elliptic-pde-2}
    f\left(|\nabla u|,-\nabla^2 u\right)=b(\cdot, u,|\nabla u|) \quad \text { in } \Omega .
\end{align}
Suppose the function $f: [0,\infty) \times \mathcal{D}_{\bar{\Gamma}}$ has the following properties:
\begin{enumerate}
    \item For every $p \in[0, \infty), f(p, \cdot)$ is an isotropic function on $\mathcal{D}_{\bar{\Gamma}}$ (see Remark \ref{rmk:defn-D} for the definitions of isotropic and $\mathcal{D}_{\bar{\Gamma}}$),
    \item $f$ is non-decreasing in the first variable,
$$
p \leq q \Rightarrow f(p, \cdot) \leq f(q, \cdot),
$$
    \item non-decreasing in the second variable,
$$
\kappa_i \leq \lambda_i \quad \forall 1 \leq i \leq n \quad \Rightarrow \quad f\left(\cdot, \operatorname{diag}\left(\kappa_1, \ldots, \kappa_n\right)\right) \leq f\left(\cdot, \operatorname{diag}\left(\lambda_1, \ldots, \lambda_n\right)\right),
$$
\item convex in the second variable,
$$
f\left(\cdot, \frac{1}{2}(A+B)\right) \leq \frac{1}{2}(f(\cdot, A)+f(\cdot, B)) \quad \forall A, B \in \mathcal{D}_{\bar{\Gamma}} .
$$
\end{enumerate}
Suppose the function $b$ is
\begin{enumerate}
    \item strictly decreasing in the second variable,
    \item non-increasing in the third variable,
    \item jointly concave in the first two variables:
$$
b\left(\gamma_{x, y}(0), \frac{1}{2}(u(x)+u(y)), p\right) \geq \frac{1}{2} b(x, u(x), p)+\frac{1}{2} b(y, u(y), p),
$$
where $\gamma_{x,y}:[-1,1]\to \bar{\Omega}$ is the unique minimizing geodesic such that $\gamma_{x,y}(-1) = x$ and $\gamma_{x,y}(1) = y$.
\end{enumerate}
If furthermore for all $(x, y) \in \partial \Omega \times \bar{\Omega}$ there holds
\EQ{\label{eq:mat-bdry-condition}
D u_x\left(\dot{\gamma}_{x, y}(-1)\right)-D u_y\left(\dot{\gamma}_{x, y}(1)\right)>0,
}
then $u$ is concave.
\end{theorem}
Furthermore, as in \cite{Korevaar-Convex} we also establish a parabolic analog of Theorem \ref{thm:elliptic}.
\begin{theorem}[Parabolic concavity principle]\label{thm:parabolic}
Let $\mathcal{M}$ be a locally symmetric space of dimension $n \geq 1$ with sectional curvatures lying in the interval $[0,A]$. Let $\Omega \subset \mathcal{M}$ be a domain with geodesically convex closure and diameter strictly less than $\frac{\pi}{\sqrt{A}}$.
Suppose $u \in C^2(\Omega \times (0,T]) \cap C^1(\bar{\Omega} \times (0,T])$ solves
\begin{align} \label{eq:parabolic-pde}
    \partial_t u = -f(|\nabla u|,-\nabla^2 u) + b(\cdot,u,|\nabla u|) \quad \text{in } \Omega \times (0,T],
\end{align}
where $f$ and $b$ satisfy the same conditions as in Theorem \ref{thm:elliptic}. If $u(\cdot, 0)$ is concave and for all $(x, y, t) \in \partial \Omega \times \bar{\Omega}\times (0,T]$ there holds
\EQ{\label{eqn:parabolic-bcs}
D u_x\left(\dot{\gamma}_{x, y}(-1),t\right)-D u_y\left(\dot{\gamma}_{x, y}(1),t\right)>0,
}
then $u(\cdot,t)$ is concave for all $t\in (0,T]$.
\end{theorem}
\begin{remark}[Isotropic functions and the domain of $f$]\label{rmk:defn-D}
Let $\mathcal{V}$ be an $n$-dimensional real inner product space and $\mathrm{Sym}(\mathcal{V})$ the space of self-adjoint endomorphisms of $\mathcal{V}$.
A function $f: \mathrm{Sym}(\mathcal{V}) \to \R$ is called \emph{isotropic} if 
\EQ{\label{defn:isotropic}
f(\mathcal{W})=f\left(V^{-1} \circ \mathcal{W} \circ V\right)
}
for all $\mathcal{W} \in \mathrm{Sym}(\mathcal{V})$ and $V \in O(\mathcal{V})$. Consequently, the value of $f$ on $\mathcal{W} \in \mathrm{Sym}(\mathcal{V})$ depends only on the ordered set of eigenvalues $\{\kappa_1 \leq \ldots \leq \kappa_n\}$ of $\mathcal{W}$:
\begin{align}
    f(\mathcal{W}) = \tilde{f}\left(\kappa_1, \ldots, \kappa_n\right).
\end{align}
The domain $\mathcal{D}_{\bar{\Gamma}}$ of the functions $f(p,\cdot)$ in Theorem \ref{thm:elliptic} is defined as follows.
For a symmetric, open and convex cone $\Gamma \subset \R^n$ which contains the positive cone $\Gamma_{+}=\left\{\kappa \in \mathbb{R}^n: \kappa_i>0 \;\; \forall 1 \leq i \leq n\right\}$, we define $\mathcal{D}_{\bar{\Gamma}} \subset \operatorname{Sym}(\mathbb{R}^n)$ as the set of symmetric matrices with eigenvalues in $\bar{\Gamma}$.
\end{remark}

\begin{remark}
Examples of $f$ satisfying the assumptions of Theorem \ref{thm:elliptic} include the usual Laplacian, the trace of the matrix exponential, and certain weighted traces of the Hessian:
\begin{align}
    f(|\nabla u|,-\nabla^2 u) &= -\Delta u, \\
    f(|\nabla u|,-\nabla^2 u) &= \operatorname{tr}\left(\exp(-\nabla^2 u)\right), \\
    f(|\nabla u|,-\nabla^2 u) &= -\sum_{i=1}^n \lambda_i \kappa_i,
\end{align}
where $\{\kappa_1 \leq \ldots \leq \kappa_n\}$ is the ordered set of eigenvalues of $\nabla^2 u$, and $\lambda_i$ are constants with $\lambda_1 \leq \ldots \leq \lambda_n$.
\end{remark}

\subsection{Background} \label{subsec:background}
Korevaar in \cites{Korevaar-Capillary,Korevaar-Convex} showed that if $u$ solves
\begin{align}
    a^{ij}(\nabla u)\nabla_i \nabla_j u = b(x,u,\nabla u)
\end{align}
on a convex domain $\Omega \subset \R^n$ with appropriate conditions on $a^{ij}, b$ and boundary conditions on $u$, then $u$ is concave. He used this to reprove the classical Brascamp--Lieb result that the first eigenfunction of a convex domain in $\R^n$ is log-concave. Korevaar's results were later improved by Andrews and Clutterbuck \cite{andrews} who obtained a sharp log-concavity estimate for the first eigenfunction and used this to prove the fundamental gap conjecture for convex domains in $\R^n$.

Concavity principles have also been widely studied on Riemannian manifolds of constant sectional curvature. Langford and Scheuer \cite{langford2021concavity} proved a concavity principle for certain PDEs on $\Sph^n$ which, as mentioned earlier, our work further extends to locally symmetric space. Log-concavity of the first eigenfunction and fundamental gap estimates for convex domains in $\Sph^n$ have been established separately in \cites{seto,he,dai2021fundamental}. Other works establishing concavity principles on constantly-curved manifolds include for instance \cites{kennington,Kawohl,leefundamentalgap,shihcounterexample,maconvexity,leeparabolic}. Quasiconcavity (i.e. convexity of upper level sets) is another continually active topic: see e.g. \cites{rosayrudin,papadimitrakis,wanggeometric,maconvexity2,weinkove} and references therein, where the spaces considered are again constantly-curved.

Much less is known for manifolds of non-constant sectional curvature. The difficulty of extending concavity principles to such manifolds was already recognized by Korevaar in 1985 \cite{korevaar-book}, yet progress on this front has stalled until recently. In \cite{ishige2022power}, concavity principles for PDEs are proved for rotationally symmetric convex domains in Riemannian manifolds. For positively curved surfaces, log-concavity of the first eigenfunction has been shown using a barrier argument in \cite{surface-barrier} and a two-point maximum principle in \cite{surface-two-point}, from which fundamental gap estimates are derived. The very recent paper \cite{khan2024concavity} proves concavity for solutions of eigenvalue problems where the PDEs live on small conformal deformations of constantly-curved manifolds. In contrast, our result does not require either symmetry of the domain (such as rotational symmetry) or any special deformation of the metric, and does not impose restrictions on dimension. However, a limitation of our result is that it cannot be used to establish the log-concavity of the first eigenfunction on the sphere or more generally, on any locally symmetric space.

\subsection{Proof sketch} The proof of Theorem \ref{thm:elliptic} follows the framework of \cite{langford2021concavity}, where a maximum principle is established for the two-point function
\EQ{\label{defn:Z}
Z: \bar{\Omega} \times \bar{\Omega} & \rightarrow \mathbb{R}, \\
(x, y) & \mapsto u\left(\gamma_{x, y}(0)\right)-\frac{1}{2}\left(u\left(\gamma_{x, y}(-1)\right)+u\left(\gamma_{x, y}(1)\right)\right),
}
defined along the unique minimizing geodesic $\gamma_{x, y}$ joining any two points $x,y\in \bar{\Omega}.$ This requires differentiating $Z$ twice which in turn requires taking second spatial derivatives of the map
\begin{align}
\Gamma: \bar{\Omega} \times \bar{\Omega} \times[-1,1] & \rightarrow \bar{\Omega}, \\ (x, y, t) & \mapsto \gamma_{x, y}(t),
\end{align}
i.e. taking spatial derivatives of Jacobi fields along geodesics. A key fact needed in the proof
is that certain linear combinations of the second variations vanish at the midpoint $t=0$. On $\Sph^n$, it was shown in \cite{langford2021concavity} that most of these linear combinations vanish identically, making the others easier to analyze. On a general locally symmetric space $\mathcal{M}$ with nonnegative sectional curvature, most linear combinations do not vanish identically. However, we find that the locally symmetric condition yields a symmetry principle for these linear combinations, which in turn implies their vanishing at $t=0$. This combined with the assumptions of Theorem \ref{thm:elliptic} implies that $Z\geq 0$, which by a theorem of Jensen \cite{jensen1906fonctions} implies that $u$ is concave.

\subsection{Organization} The paper is organized as follows. In \S\ref{sec:cp-prelim} we recall various geometric facts about locally symmetric spaces and compute spatial derivatives of Jacobi fields along its geodesics. In \S\ref{sec:concavity-proof} we prove Theorems \ref{thm:elliptic} and \ref{thm:parabolic}. In \S\ref{sec:applications}, we provide ways to relax some of the assumptions in Theorems \ref{thm:elliptic} and \ref{thm:parabolic} and give examples of PDEs that are covered by our theorems.

\section{Jacobi fields on locally symmetric spaces} \label{sec:cp-prelim}

Let $(\mathcal{M}^n,\inner{\cdot}{\cdot})$ be a smooth Riemannian manifold of dimension $n \geq 1$ and $\nabla$ its Levi-Civita connection. For the Riemann curvature tensor, we take the convention
\begin{align}
    R(X,Y)Z = \nabla_X \nabla_Y Z - \nabla_Y \nabla_X Z - \nabla_{[X,Y]}Z.
\end{align}
We assume that $(\mathcal{M}^n,\inner{\cdot}{\cdot})$ is a locally symmetric space, i.e. $\nabla R = 0$. We also assume its sectional curvatures lie in the interval $[0,A]$, i.e. $0 \leq R(X,Y,Y,X) \leq A$ for all $x \in \mathcal{M}$ and orthonormal tangent vectors $X,Y \in T_x\mathcal{M}$.





\subsection{Setup and frame fixing} \label{subsec:setup}

We borrow the setup of \cite{langford2021concavity}. Let $\Omega \subset \mathcal{M}^n$ be a domain with geodesically convex closure $\bar{\Omega}$ and diameter $\mathrm{diam}(\Omega) < \frac{\pi}{\sqrt{A}}$. That is, each pair of distinct points in $\bar{\Omega}$ is joined by a unique minimizing geodesic in $\mathcal{M}^n$ and this geodesic segment is contained in $\bar{\Omega}$. Define
\begin{align}
    \Gamma: \bar{\Omega} \times \bar{\Omega} \times [-1,1] \to \bar{\Omega}, \quad (x,y,t) \mapsto \gamma_{x,y}(t),
\end{align}
where $\gamma_{x,y}$ is the unique minimizing geodesic with $\gamma_{x,y}(-1) = x$ and $\gamma_{x,y}(1) = y$. Hereafter we will consider the domain $\Omega$ and the points $x,y \in \bar{\Omega}$ as being fixed.
Define the Jacobi fields $\J_{x^\alpha}$ and $\J_{y^\alpha}$ along $\gamma_{x,y}$ by
\begin{equation} \label{eq:J-defns}
    \J_{x^\alpha}(t) = \frac{\del\Gamma}{\del x^\alpha}(x,y,t) \quad \text{and} \quad \J_{y^\alpha}(t) = \frac{\del\Gamma}{\del y^\alpha}(x,y,t),
\end{equation}
where the coordinate systems $(x^\alpha)$ near $x$ and $(y^\alpha)$ near $y$ are chosen later. The Jacobi fields $\J_{x^\alpha}$ and $\J_{y^\alpha}$ satisfy
\begin{equation} \label{eq:J-bdry-vals}
    \J_{x^\alpha}(-1) = \frac{\del}{\del x^\alpha}\Big|_x, \quad \J_{x^\alpha}(1) = 0, \quad \J_{y^\alpha}(-1) = 0, \quad \J_{y^\alpha}(1) = \frac{\del}{\del y^\alpha}\Big|_y.
\end{equation}
We also need the first spatial derivatives of $\J_{x^\alpha}$ and $\J_{y^\alpha}$ which are vector fields along $\gamma_{x,y}$ defined by
\begin{align}
    \begin{aligned} \label{eq:K}
        K_{x^\alpha x^\beta}(t) &= \nabla_{\del_{x^\beta}}\frac{\del\Gamma}{\del x^\alpha}(x,y,t), \\
        K_{x^\alpha y^\beta}(t) &= \nabla_{\del_{y^\beta}}\frac{\del\Gamma}{\del x^\alpha}(x,y,t), \\
        K_{y^\alpha x^\beta}(t) &= \nabla_{\del_{x^\beta}}\frac{\del\Gamma}{\del y^\alpha}(x,y,t), \\
        K_{y^\alpha y^\beta}(t) &= \nabla_{\del_{y^\beta}}\frac{\del\Gamma}{\del y^\alpha}(x,y,t).
    \end{aligned}
\end{align}
We have the symmetries $K_{x^\alpha x^\beta} = K_{x^\beta x^\alpha}$, $K_{x^\alpha y^\beta} = K_{y^\beta x^\alpha}$, etc.
Since $\Gamma(\cdot,\cdot,-1)$ and $\Gamma(\cdot,\cdot,1)$ are projections onto the first and second factors respectively, their second covariant derivatives vanish identically and hence
\begin{equation} \label{eq:K-bdry}
    K_{x^\alpha x^\beta}(\pm 1) = K_{x^\alpha y^\beta}(\pm 1) = K_{y^\alpha x^\beta}(\pm 1) = K_{y^\alpha y^\beta}(\pm 1) = 0.
\end{equation}
Since $x$ and $y$ are considered fixed, we will generally consider $\Gamma$ as a function of only time, i.e. $\Gamma(t):= \Gamma(x,y,t)$. Derivatives with respect to $t$ will be indicated using dots. With our conventions, $\J_{x^\alpha}$ and $\J_{y^\alpha}$ satisfy the Jacobi field equations
\begin{align}
    \ddot{\J}_{x^\alpha}(t) +  R(\J_{x^\alpha},\dot{\Gamma})\dot{\Gamma} &= 0, \\
    \ddot{\J}_{y^\alpha}(t) +  R(\J_{y^\alpha},\dot{\Gamma})\dot{\Gamma} &= 0.
\end{align}
We now fix a convenient frame to work in. Let $E_1 = \frac{\dot{\Gamma}}{|\dot{\Gamma}|}$, which has unit length and is parallel along $\Gamma$. Since $R(\cdot,\dot{\Gamma})\dot{\Gamma}$ is self-adjoint with respect to $\inner{\cdot}{\cdot}$, and $(E_1)_x \in T_x\mathcal{M}$ belongs to its kernel, we can complete it to an orthonormal basis $\{E_1,\ldots,E_n\}$ for $T_x\mathcal{M}$ that diagonalizes $R(\cdot,\dot{\Gamma})\dot{\Gamma}$ at the point $x$. Then there are numbers $\kappa_1, \ldots, \kappa_n \in \R$ such that at the point $x$,
\begin{align} \label{eq:kappa-defn}
    \inner{R(E_\alpha, \dot{\Gamma})\dot{\Gamma}}{E_\beta} = \kappa_\alpha|\dot{\Gamma}|^2 \delta_{\alpha\beta}.
\end{align}
Note that $\kappa_1 = 0$, and for $\alpha \geq 2$, the number $\kappa_\alpha$ is the sectional curvature of the 2-plane spanned by $E_1, E_\alpha \in T_x\mathcal{M}$. By the assumptions on sectional curvature, $0 \leq \kappa_\alpha \leq A$.

Now, extend the basis $\{E_1,\ldots,E_n\}$ along $\Gamma$ by parallel transport and continue to denote the resulting orthonormal frame by $\{E_1,\ldots,E_n\}$. Since $\nabla R = 0$ by assumption, and the vector fields $E_\alpha$ and $\dot{\Gamma}$ are parallel along $\Gamma$, the identity \eqref{eq:kappa-defn} remains true at all points on $\Gamma$, with the same numbers $\kappa_\alpha$.

Let $(x^\alpha)_{\alpha=1}^n$ be Riemannian normal coordinates at $x$ such that $(\frac{\del}{\del x^1},\frac{\del}{\del x^{\bar{1}}}, \ldots, \frac{\del}{\del x^{n}},\frac{\del}{\del x^{\bar{n}}})$ coincides with $\{E_1,\ldots,E_n\}$ at $x$.
Likewise, let $(y^\alpha)_{\alpha=1}^n$ be normal coordinates at $y$ such that $(\frac{\del}{\del y^1},\frac{\del}{\del y^{\bar{1}}}, \ldots, \frac{\del}{\del y^{n}},\frac{\del}{\del y^{\bar{n}}})$ coincides with $\{E_1,\ldots,E_n\}$ at $y$.
Thus
\begin{equation} \label{eq:Ei-bdry}
    E_\alpha(-1) = \frac{\del}{\del x^\alpha}\Big|_x, \quad E_\alpha(1) = \frac{\del}{\del y^\alpha}\Big|_y
\end{equation}
for all $\alpha=1,\ldots,n$.


With frames and coordinates thus chosen, we now write down expressions for the Jacobi fields $\J_{x^\alpha}$ and $\J_{y^\alpha}$ along $\Gamma$ defined in \eqref{eq:J-defns}. 
Since the sectional curvatures lie in the interval $[0,A]$, the diameter bound $\mathrm{diam}(\Omega) < \frac{\pi}{\sqrt{A}}$ implies that $\bar{\Omega}$ contains no pair of conjugate points. Thus, Jacobi fields along geodesics in $\bar{\Omega}$ are uniquely determined by their boundary values. We use this fact repeatedly in the sequel. 

\begin{lemma} \label{lem:Jacobi-explicit}
    For each $\alpha \geq 1$ we have
    \begin{gather}
        \J_{x^\alpha}(t) = v_\alpha(1-t)E_\alpha(t), \quad \J_{y^\alpha}(t) = v_\alpha(1+t) E_\alpha(t)
    \end{gather}
    where
    \begin{equation} \label{eq:valpha}
        v_\alpha(t) = \begin{cases}
            \frac{t}{2} & \text{if } \alpha=1, \\  \frac{\sin(\sqrt{\kappa_\alpha}|\dot{\Gamma}|t)}{\sin(2\sqrt{\kappa_\alpha}|\dot{\Gamma}|)} & \text{if } \alpha \geq 2.
        \end{cases}
    \end{equation}
\end{lemma}
\begin{proof}
    Since the $E_\beta$'s are parallel along $\Gamma$ and diagonalize $R(\cdot,\dot{\Gamma})\dot{\Gamma}$, we have for each $\beta \geq 1$
    \begin{align}
        \frac{d^2}{dt^2} \inner{\J_{x^\alpha}}{E_\beta} &= \inner{\ddot{\J}_{x^\alpha}}{E_\beta} = -\inner{R(\J_{x^\alpha},\dot{\Gamma})\dot{\Gamma}}{E_\beta} = -\inner{\J_{x^\alpha}}{E_\beta} \inner{R(E_\beta,\dot{\Gamma})\dot{\Gamma}}{E_\beta} = -\kappa_\beta \inner{\J_{x^\alpha}}{E_\beta}.
    \end{align}
    Solving this ODE for $\inner{\J_{x^\alpha}}{E_\beta}$ with boundary values $\inner{\J_{x^\alpha}}{E_\beta}(-1) = \delta_{\alpha\beta}$ and $\inner{\J_{x^\alpha}}{E_\beta}(1) = 0$ gives the claimed formulas for $\J_{x^\alpha}(t)$. The formulas for $\J_{y^\alpha}(t)$ are proved similarly.
\end{proof}

\subsection{Spatial derivatives of Jacobi fields}

The following combinations of the $K(t)$ vector fields defined in \eqref{eq:K} will appear when computing the Hessian of the concavity function $Z$:
\begin{align}
\begin{aligned} \label{eq:Kpm}
    \mathcal{K}^+_{\alpha\beta} &:= K_{x^\alpha x^\beta} + K_{y^\alpha y^\beta} + K_{x^\alpha y^\beta} + K_{y^\alpha x^\beta}, \\
    \mathcal{K}^-_{\alpha\beta} &:= K_{x^\alpha x^\beta} + K_{y^\alpha y^\beta} - K_{x^\alpha y^\beta} - K_{y^\alpha x^\beta}.
\end{aligned}
\end{align}
Note that $\mathcal{K}^\pm_{\alpha\beta} = \mathcal{K}^\pm_{\beta\alpha}$.
The aim of this subsection is to prove the following key fact.
\begin{proposition} \label{prop:K-vanishing}
    For each $\alpha,\beta \geq 1$, we have $\mathcal{K}^\pm_{\alpha\beta}(0) = 0$.
\end{proposition}
We begin in the same way as \cite{langford2021concavity}, namely by writing down an ODE satisfied by the $K(t)$ vector fields. The proof only uses that $\nabla R = 0$ and hence still applies in our setting.
\begin{lemma}[\cite{langford2021concavity}*{Lemma 3.1}] \label{lem:K-ODE}
    For all $\alpha,\beta \geq 1$ there holds
    \begin{align}
        0 = \ddot{K}_{x^\alpha x^\beta} + R(K_{x^\alpha x^\beta},\dot{\Gamma})\dot{\Gamma} + 2R(\J_{x^\alpha},\dot{\Gamma})\dot{\J}_{x^\beta} + 2R(\J_{x^\beta},\dot{\Gamma})\dot{\J}_{x^\alpha},
    \end{align}
    and similarly for $K_{y^\alpha y^\beta}$, $K_{x^\alpha y^\beta}$ and $K_{y^\alpha x^\beta}$.
\end{lemma}

In \cite{langford2021concavity}, it is shown using the formula for the curvature tensor on $\Sph^n$ that the vector fields $\sum_{\alpha,\beta\geq 1}\mathcal{K}^\pm_{\alpha\beta}\xi^\alpha \xi^\beta$, defined for each $\xi \in \R^n$, all point in the $\dot{\Gamma}$ direction. The authors then show that $\sum_{\alpha,\beta\geq 1}\mathcal{K}^\pm_{\alpha\beta}(0)\xi^\alpha \xi^\beta=0$ for all $\xi \in \R^n$. On $\mathcal{M}$, the directionality result is not true, but the midpoint-vanishing result is true with an even stronger conclusion (Proposition \ref{prop:K-vanishing}). All that one has to do is directly work with the $\mathcal{K}^\pm_{\alpha\beta}$'s instead of contracting with vectors $\xi$, as we do below.
\begin{lemma} \label{lem:fundamental}
    For any $\alpha,\beta,\gamma \geq 1$ and at any time $t \in [-1,1]$, we have
    \begin{align}
        &\del_t^2 \left( \inner{\mathcal{K}^+_{\alpha\beta}}{E_\gamma}E_\gamma \right) + R\left( \inner{\mathcal{K}^+_{\alpha\beta}}{E_\gamma}E_\gamma, \dot{\Gamma} \right) \dot{\Gamma} \\
        &= -2(v_\alpha(1-t)+v_\alpha(1+t))(-\dot{v}_\beta(1-t) + \dot{v}_\beta(1+t)) \inner{R(E_\alpha, \dot{\Gamma})E_\beta}{E_\gamma}E_\gamma \\
        &\quad -2(v_\beta(1-t)+v_\beta(1+t))(-\dot{v}_\alpha(1-t) + \dot{v}_\alpha(1+t))\inner{R(E_\beta, \dot{\Gamma})E_\alpha}{E_\gamma}E_\gamma \label{eqn:pos-kappa},
    \end{align}
    and
    \begin{align}
        &\del_t^2 \left( \inner{\mathcal{K}^-_{\alpha\beta}}{E_\gamma}E_\gamma \right) + R\left( \inner{\mathcal{K}^-_{\alpha\beta}}{E_\gamma}E_\gamma, \dot{\Gamma} \right) \dot{\Gamma} \\
        &= 2(v_\alpha(1-t)-v_\alpha(1+t))(\dot{v}_\beta(1-t)+\dot{v}_\beta(1+t)) \inner{R(E_\alpha, \dot{\Gamma})E_\beta}{E_\gamma}E_\gamma \\
        &\quad +2(v_\beta(1-t)-v_\beta(1+t))(\dot{v}_\alpha(1-t)+\dot{v}_\alpha(1+t))\inner{R(E_\beta, \dot{\Gamma})E_\alpha}{E_\gamma}E_\gamma. \label{eqn:neg-kappa}
    \end{align}
\end{lemma}
\begin{proof}
    Recall that the $E_\gamma$'s are parallel and diagonalize $R(\cdot,\dot{\Gamma})\dot{\Gamma}$. This together with Lemma \ref{lem:K-ODE} give
    \begin{align}
        &\del_t^2 \left( \inner{\mathcal{K}^\pm_{\alpha\beta}}{E_\gamma}E_\gamma \right) + R\left( \inner{\mathcal{K}^\pm_{\alpha\beta}}{E_\gamma}E_\gamma, \dot{\Gamma} \right) \dot{\Gamma} \\
        &= \inner{\ddot{\mathcal{K}}^\pm_{\alpha\beta}}{E_\gamma}E_\gamma + \inner{\mathcal{K}^\pm_{\alpha\beta}}{E_\gamma} R(E_\gamma,\dot{\Gamma})\dot{\Gamma} \\
        &= -\inner{R(\mathcal{K}^\pm_{\alpha\beta},\dot{\Gamma})\dot{\Gamma}}{E_\gamma}E_\gamma + \inner{\mathcal{K}^\pm_{\alpha\beta}}{E_\gamma} R(E_\gamma,\dot{\Gamma})\dot{\Gamma} \\
        &\quad - 2 \inner{R(\J_{x^\alpha},\dot{\Gamma})\dot{\J}_{x^\beta}}{E_\gamma}E_\gamma - 2 \inner{R(\J_{x^\beta},\dot{\Gamma})\dot{\J}_{x^\alpha}}{E_\gamma}E_\gamma \\
        &\quad - 2 \inner{R(\J_{y^\alpha},\dot{\Gamma})\dot{\J}_{y^\beta}}{E_\gamma}E_\gamma - 2 \inner{R(\J_{y^\beta},\dot{\Gamma})\dot{\J}_{y^\alpha}}{E_\gamma}E_\gamma \\
        &\quad \mp 2 \inner{R(\J_{x^\alpha},\dot{\Gamma})\dot{\J}_{y^\beta}}{E_\gamma}E_\gamma \mp 2 \inner{R(\J_{y^\beta},\dot{\Gamma})\dot{\J}_{x^\alpha}}{E_\gamma}E_\gamma \\
        &\quad \mp 2 \inner{R(\J_{y^\alpha},\dot{\Gamma})\dot{\J}_{x^\beta}}{E_\gamma}E_\gamma \mp 2 \inner{R(\J_{x^\beta},\dot{\Gamma})\dot{\J}_{y^\alpha}}{E_\gamma}E_\gamma \\
        &= -\inner{\mathcal{K}^\pm_{\alpha\beta}}{E_\gamma} \inner{R(E_\gamma,\dot{\Gamma})\dot{\Gamma}}{E_\gamma}E_\gamma + \inner{\mathcal{K}^\pm_{\alpha\beta}}{E_\gamma} \inner{R(E_\gamma,\dot{\Gamma})\dot{\Gamma}}{E_\gamma}E_\gamma \\
        &\quad -2\inner{R(\J_{x^\alpha} \pm \J_{y^\alpha}, \dot{\Gamma})(\dot{\J}_{x^\beta} \pm \dot{\J}_{y^\beta})}{E_\gamma}E_\gamma - 2\inner{R(\J_{x^\beta} \pm \J_{y^\beta}, \dot{\Gamma})(\dot{\J}_{x^\alpha} \pm \dot{\J}_{y^\alpha})}{E_\gamma}E_\gamma.
    \end{align}
    The lemma follows from this and Lemma \ref{lem:Jacobi-explicit}.
\end{proof}

\begin{proof}[Proof of Proposition \ref{prop:K-vanishing}]
    Let $\alpha,\beta,\gamma \geq 1$. Since $\nabla R = 0$, and the vector fields $E_\alpha$ and $\dot{\Gamma}$ are parallel along $\Gamma$, it follows that coefficients of the form $\inner{R(E_\alpha,\dot{\Gamma})E_\beta}{E_\gamma}$ are constant in $t$. We denote these constants by $c_{\alpha\beta\gamma}$. Then \eqref{eqn:pos-kappa} implies
    \begin{align}
        &\left[ \del_t^2 \inner{\mathcal{K}^+_{\alpha \beta}}{E_{\gamma}} + |\dot{\Gamma}|^2 \inner{\mathcal{K}^+_{\alpha \beta}}{E_{\gamma}}\right] E_{\gamma} \\
        &=\del_t^2 \left( \inner{\mathcal{K}^+_{\alpha \beta}}{E_{\gamma}}E_{\gamma} \right) +  R\left( \inner{\mathcal{K}^+_{\alpha \beta}}{E_{\gamma}}E_{\gamma}, \dot{\Gamma} \right) \dot{\Gamma} \\
        &= -2c_{\alpha\beta\gamma}(v_\alpha(1-t)+v_\alpha(1+t))(-\dot{v}_{\beta}(1-t) + \dot{v}_{\beta}(1+t)) E_{\gamma} \\
        &\quad -2c_{\beta \alpha\gamma}(v_{\beta}(1-t)+v_{\beta}(1+t))(-\dot{v}_\alpha(1-t) + \dot{v}_\alpha(1+t)) E_{\gamma}. \label{eq:ccc}
    \end{align}
    This computation shows that the function $\phi(t) := \inner{\mathcal{K}^+_{\alpha\beta}}{E_{\gamma}}$ satisfies
    \begin{equation} \label{eq:phi-ODE}
        \ddot{\phi} + |\dot{\Gamma}|^2 \phi + \eta = 0,
    \end{equation}
    where
    \begin{align}
        \eta(t) &=c_{\alpha\beta\gamma}|\dot{\Gamma}|(v_\alpha(1-t)+v_\alpha(1+t))(-\dot{v}_{\beta}(1-t) + \dot{v}_{\beta}(1+t)) \\
        &\quad + c_{\beta \alpha\gamma}|\dot{\Gamma}|(v_{\beta}(1-t)+v_{\beta}(1+t))(-\dot{v}_\alpha(1-t) + \dot{v}_\alpha(1+t)).
    \end{align}
    But $\eta(-t) = -\eta(t)$, so the function $\psi(t) := \phi(-t)$ satisfies
    \begin{equation}
        \ddot{\psi} + |\dot{\Gamma}|^2 \psi - \eta = 0.
    \end{equation}
    Adding this to \eqref{eq:phi-ODE} gives the ODE
    \begin{equation}
        (\ddot{\phi} + \ddot{\psi}) + |\dot{\Gamma}|^2 (\phi + \psi) = 0.
    \end{equation}
    But $\phi$ and $\psi$ both vanish at $t = \pm 1$, so $\phi + \psi$ vanishes identically. Evaluating at $t=0$ gives $\inner{\mathcal{K}^+_{\alpha\beta}}{E_{\gamma}} = 0$. Since the direction $E_\gamma$ was arbitrary, this shows that $\mathcal{K}^+_{\alpha \beta}$ vanishes at $t=0.$ Next, using \eqref{eqn:neg-kappa} we get
    \begin{align}
        \left(\del_t^2 \inner{\mathcal{K}^-_{\alpha \beta}}{E_\gamma} + |\dot{\Gamma}|^2 \inner{\mathcal{K}^-_{\alpha \beta}}{E_\gamma} \right)E_\gamma &=\del_t^2 \left( \inner{\mathcal{K}^-_{\alpha \beta}}{E_\gamma}E_\gamma \right) +  R\left( \inner{\mathcal{K}^-_{\alpha \beta}}{E_\gamma}E_\gamma, \dot{\Gamma} \right) \dot{\Gamma} \\
         &= 2c_{\alpha \beta \gamma}(v_\alpha(1-t)-v_\alpha(1+t))(\dot{v}_\beta(1-t)+\dot{v}_\beta(1+t)) E_\gamma \\
        &\quad +2c_{\beta \alpha \gamma}(v_\beta(1-t)-v_\beta(1+t))(\dot{v}_\alpha(1-t)+\dot{v}_\alpha(1+t))E_\gamma.
    \end{align}
    Thus $\phi := \inner{\mathcal{K}^-_{\alpha \beta}}{E_\gamma}$ satisfies $\ddot{\phi}+|\dot{\Gamma}|^2\phi = \eta$, where $\eta$ is an odd function. Now arguing as before we get that $\mathcal{K}^-_{\alpha \beta}$ vanishes at $t=0$.
\end{proof}


\section{Concavity principles}\label{sec:concavity-proof}

\subsection{Proofs of main theorems}

We now use the results of the previous section to prove Theorems \ref{thm:elliptic} and \ref{thm:parabolic}. The arguments are essentially due to \cite{langford2021concavity} which is in turn based on \cite{Korevaar-Convex}, but due to notational differences we rewrite them here for the reader's convenience.

Let $\Omega \subset \mathcal{M}$ be a domain such that $\bar{\Omega}$ is geodesically convex, and recall the two-point function from \eqref{defn:Z}:
\begin{align}
    Z: \bar{\Omega} \times \bar{\Omega} & \rightarrow \mathbb{R}, \\
    (x, y) & \mapsto u\left(\gamma_{x, y}(0)\right)-\frac{1}{2}\left(u\left(\gamma_{x, y}(-1)\right)+u\left(\gamma_{x, y}(1)\right)\right).
\end{align}
Fix an interior minimum point $(x,y) \in \Omega \times \Omega$ of $Z$ (assuming for now that one exists) and let $z = \gamma_{x,y}(0)$ be the geodesic midpoint. Let $\{E_1,\ldots,E_n\}$ be the parallel orthonormal frame along $\gamma_{x,y}$ constructed in \S\ref{subsec:setup}. Using this frame we identify the tangent spaces to $\mathcal{M}$ at $x$, $y$ and $z$ each with $\R^n$. With this setup we have
\begin{lemma} \label{lem:first-order-Zmin}
    If $(x,y) \in \Omega \times \Omega$ is an interior critical point of $Z$, and $z = \gamma_{x,y}(0)$, then after identifying tangent spaces as above we have
    \begin{align}
        \nabla u(z) = V(\nabla u(x)) = V(\nabla u(y)),
    \end{align}
    where $V = \mathrm{diag}\left(\frac{1}{2v_\alpha(1)}\right)_{\alpha=1,\ldots,n}$.
\end{lemma}
\begin{proof}
    Differentiating $Z$ in the directions $E_\alpha \oplus 0$ and $0 \oplus E_\alpha$ at $(x,y)$, and using Lemma \ref{lem:Jacobi-explicit}, we obtain respectively for each $\alpha \geq 1$
    \begin{align}
        0 &= \nabla_{E_\alpha \oplus 0}Z = \nabla_{\J_{x^\alpha}(0)} u(z) - \frac{1}{2}\nabla_{E_\alpha} u(x) = v_\alpha(1) \nabla_{E_\alpha}u(z) - \frac{1}{2}\nabla_{E_\alpha} u(x), \\
        0 &= \nabla_{0 \oplus E_\alpha}Z = \nabla_{\J_{y^\alpha}(0)} u(z) - \frac{1}{2}\nabla_{E_\alpha} u(y) = v_\alpha(1) \nabla_{E_\alpha}u(z) - \frac{1}{2}\nabla_{E_\alpha} u(y).
    \end{align}
    Thus $\nabla_{E_\alpha}u(z) = \frac{1}{2v_\alpha(1)}\nabla_{E_\alpha}u(x) = \frac{1}{2v_\alpha(1)}\nabla_{E_\alpha}u(y)$ for each $\alpha \geq 1$. The lemma follows.
\end{proof}
\begin{remark} \label{rmk:Vcontracting}
    The expressions \eqref{eq:valpha} for $v_\alpha$ imply that $\frac{1}{2v_\alpha(1)} \leq 1$ for each $\alpha \geq 1$. Thus $V \leq I$, i.e. $V$ is a weakly contracting map on $\R^n$. Hence $|\nabla u(z)| \leq |\nabla u(x)| = |\nabla u(y)|$.
\end{remark}
\begin{lemma} \label{lem:second-order-Zmin}
    If $(x,y) \in \Omega \times \Omega$ is a local minimum of $Z$, and $z = \gamma_{x,y}(0)$, then in the sense of bilinear forms on $\R^n$ we have
    \begin{equation} \label{eq:second-order-1}
        -(\nabla^2 u)_z \leq -\frac{1}{2}((\nabla^2 u)_x + (\nabla^2 u)_y)(V(\cdot),V(\cdot))
    \end{equation}
    and
    \begin{equation} \label{eq:second-order-2}
        (\nabla^2 u)_x + (\nabla^2 u)_y \leq 0,
    \end{equation}
    where $V = \mathrm{diag}\left(\frac{1}{2v_\alpha(1)}\right)_{\alpha=1,\ldots,n}$.
\end{lemma}
\begin{proof}
    We compute the Hessian of $Z$ at an interior point $(x,y) \in \Omega \times \Omega$. Let $W = a^\alpha \del_{x^\alpha} + b^\alpha \del_{y^\alpha} \in \Gamma(T(\Omega \times \Omega))$ be a vector field with $(a^1,\ldots,a^n),(b^1,\ldots,b^n) \in \R^n$ constant. Then in normal coordinates at $(x,y)$ we have
    \begin{align}
        (\nabla^2 Z)(W,W) &= \nabla_W \nabla_W Z = (a^\alpha \del_{x^\alpha} + b^\alpha \del_{y^\alpha}) (a^\beta \del_{x^\beta} + b^\beta \del_{y^\beta}) Z(x,y) \\
        &= (a^\alpha \del_{x^\alpha} + b^\alpha \del_{y^\alpha}) \cdot \left[ a^\beta \left(\nabla u|_z(\J_{x^\beta}) - \frac{1}{2}\del_{x^\beta}u(x)\right) + b^\beta \left( \nabla u|_z(\J_{y^\beta}) - \frac{1}{2}\del_{y^\beta} u(y) \right) \right] \\
        &= a^\alpha a^\beta \left[ \nabla^2 u|_z(\J_{x^\alpha},\J_{x^\beta}) + \nabla u|_z(K_{x^\alpha x^\beta}) - \frac{1}{2}\del_{x^\alpha}\del_{x^\beta} u(x) \right] \\
        &\quad + a^\alpha b^\beta \left[ \nabla^2 u|_z(\J_{x^\alpha},\J_{y^\beta}) + \nabla u|_z(K_{x^\alpha y^\beta}) \right] \\
        &\quad + b^\alpha a^\beta \left[ \nabla^2 u|_z(\J_{y^\alpha},\J_{x^\beta}) + \nabla u|_z(K_{y^\alpha x^\beta}) \right] \\
        &\quad + b^\alpha b^\beta \left[ \nabla^2 u|_z(\J_{y^\alpha},\J_{y^\beta}) + \nabla u|_z(K_{y^\alpha y^\beta} - \frac{1}{2} \del_{y^\alpha} \del_{y^\beta} u(y) \right] \\
        &= (\nabla^2 u)|_z(a^\alpha \J_{x^\alpha} + b^\alpha \J_{y^\alpha}, a^\beta \J_{x^\beta} + b^\beta \J_{y^\beta}) - \frac{1}{2}(a^\alpha a^\beta \nabla^2_{\alpha\beta} u(x) + b^\alpha b^\beta \nabla^2_{\alpha\beta} u(y)) \\
        &\quad + \nabla u|_z(a^\alpha a^\beta K_{x^\alpha x^\beta} + a^\alpha b^\beta K_{x^\alpha y^\beta} + b^\alpha a^\beta K_{y^\alpha x^\beta} + b^\alpha b^\beta K_{y^\alpha y^\beta}). \label{eq:HessZ-formula}
    \end{align}
    To prove \eqref{eq:second-order-1}, let $\xi \in \R^n$ be arbitrary and let $a = b = V\xi$. Lemma \ref{lem:Jacobi-explicit} gives
    \begin{align}
        a^\alpha \J_{x^\alpha} + b^\alpha \J_{y^\alpha} &= \sum_\alpha (V\xi)^\alpha (v_\alpha(1-t)+v_\alpha(1+t)) E_\alpha(t) = \sum_\alpha \frac{v_\alpha(1-t)+v_\alpha(1+t)}{2v_\alpha(1)} \xi^\alpha E_\alpha(t).
    \end{align}
    At the midpoint $z$ (i.e. $t=0$), all factors on the right are one and we have
    \begin{equation} \label{eq:31}
        (\nabla^2 u)|_z(a^p \J_{x^p} + b^p \J_{y^p}, a^q \J_{x^q} + b^q \J_{y^q}) = (\nabla^2 u)|_z(\xi,\xi).
    \end{equation}
    Also, we have
    \begin{align}
        a^\alpha a^\beta \nabla^2_{\alpha\beta}u(x) + b^\alpha b^\beta \nabla^2_{\alpha\beta}u(y) &= (\nabla^2 u)_x(a,a) + (\nabla^2 u)_y(b,b) \\
        &= \left((\nabla^2 u)_x + (\nabla^2 u)_y\right)(V\xi,V\xi). \label{eq:32}
    \end{align}
    Proposition \ref{prop:K-vanishing} says that all $\mathcal{K}^+_{\alpha\beta} = 0$ at $t=0$. So at the point $z$,
    \begin{align}
        a^\alpha a^\beta K_{x^\alpha x^\beta} + a^\alpha b^\beta K_{x^\alpha y^\beta} + b^\alpha a^\beta K_{y^\alpha x^\beta} + b^\alpha b^\beta K_{y^\alpha y^\beta} = \sum_{\alpha,\beta} \frac{1}{4v_\alpha(1)v_\beta(1)} \xi^\alpha \xi^\beta \mathcal{K}^+_{\alpha\beta} = 0. \label{eq:33}
    \end{align}
    Substituting \eqref{eq:31}, \eqref{eq:32} and \eqref{eq:33} into \eqref{eq:HessZ-formula}, we get
    \begin{align}
        (\nabla^2 Z)(W,W) &= (\nabla^2 u)|_z(\xi,\xi) - \frac{1}{2}\left((\nabla^2 u)_x + (\nabla^2 u)_y \right)(V\xi,V\xi).
    \end{align}
    Since $(x,y)$ is an interior local minimum of $Z$, we have $\nabla^2 Z(W,W) \geq 0$, and \eqref{eq:second-order-1} follows.

    To prove \eqref{eq:second-order-2}, let $\xi \in \R^n$ be arbitrary and let $a = -b = \xi$. Then
    \begin{equation}
        a^\alpha \J_{x^\alpha} + b^\alpha \J_{y^\alpha}  = \xi^\alpha (\J_{x^\alpha} - \J_{y^\alpha}) = \xi^\alpha (v_\alpha(1-t)-v_\alpha(1+t))E_\alpha(t)
    \end{equation}
    which vanishes at the point $z$, i.e. when $t=0$. Thus, \eqref{eq:HessZ-formula} becomes
    \begin{align}
        0 &\leq (\nabla^2 Z)(W,W) = -\frac{1}{2}((\nabla^2 u)_x + (\nabla^2 u)_y)(\xi,\xi) + \nabla u|_z\left( \sum_{\alpha,\beta} \xi^\alpha \xi^\beta \mathcal{K}^-_{\alpha\beta} \right) \\
        &= -\frac{1}{2}((\nabla^2 u)_x + (\nabla^2 u)_y)(\xi,\xi)
    \end{align}
    where the last equality is Proposition \ref{prop:K-vanishing}.
\end{proof}

\begin{proof}[Proof of Theorem \ref{thm:elliptic}]
    Let $(x,y) \in \bar{\Omega} \times \bar{\Omega}$ be a minimum for $Z$. We claim that $(x,y) \in \Omega \times \Omega$. Indeed, if this is false, say $x \in \del\Omega$ without loss of generality, then define the function $c(s) = Z(\gamma_{x,y}(-1+s), \gamma_{x,y}(1-s))$. By \eqref{eq:mat-bdry-condition} we have
    \begin{align}
        \dot{c}(0) = -Du_x(\dot{\gamma}_{x,y}(-1)) +Du_y(\dot{\gamma}_{x,y}(1)) < 0.
    \end{align}
    Thus, for small $s > 0$ the value of $Z$ at the point $(\gamma_{x,y}(-1+s),\gamma_{x,y}(1-s)) \in \Omega \times \Omega$ is strictly less than $Z(x,y)$, contradicting minimality. Hence $(x,y) \in \Omega \times \Omega$.
    
    From Lemma \ref{lem:second-order-Zmin} and \cite{langford2021concavity}*{Lemma 2.1} (to apply this correctly we use that $V$ is a contraction, see Remark \ref{rmk:Vcontracting}), we have
    \begin{align} \label{eq:15091}
        -(\nabla^2 u)_z \leq -\frac{1}{2}((\nabla^2 u)_x + (\nabla^2 u)_y)(V(\cdot),V(\cdot)) \leq -\frac{1}{2}((\nabla^2 u)_x + (\nabla^2 u)_y).
    \end{align}
    Combining with the hypotheses on $f$ and $b$ yields the following chain of inequalities:
    \begin{alignat}{2}
        b(z,u(z),|\nabla u(z)|) &= f(|\nabla u(z)|, -\nabla^2 u(z)) & \text{(PDE for $u$)} \\
        &\leq f\left(|\nabla u(z)|, -\frac{1}{2}((\nabla^2 u)_x + (\nabla^2 u)_y)\right) \quad& \text{($f \uparrow$ in second slot, and \eqref{eq:15091})} \\
        &\leq \frac{1}{2}f\left(|\nabla u(z)|, -(\nabla^2 u)_x\right) + \frac{1}{2}f\left(|\nabla u(z)|, -(\nabla^2 u)_y\right) \quad& \text{($f$ convex in second slot)} \\
        &\leq \frac{1}{2}f\left(|\nabla u(x)|, -(\nabla^2 u)_x\right) + \frac{1}{2}f\left(|\nabla u(y)|, -(\nabla^2 u)_y\right) \quad& \text{($f \uparrow$ in first slot, Remark \ref{rmk:Vcontracting})} \\
        &= \frac{1}{2}b(x,u(x),|\nabla u(x)|) + \frac{1}{2}b(y,u(y),|\nabla u(y)|) \quad& \text{(PDE for $u$)} \\
        &\leq \frac{1}{2}b(x,u(x),|\nabla u(z)|) + \frac{1}{2}b(y,u(y),|\nabla u(z)|) \quad & \text{($b \downarrow$ in third slot, Remark \ref{rmk:Vcontracting})} \\
        &\leq b\left(z, \frac{1}{2}(u(x)+u(y)), |\nabla u(z)| \right). \quad &\text{($b$ jointly concave in slots 1 \& 2)}
    \end{alignat}
    Finally, since $b$ is strictly decreasing in the second slot, it follows that $u(z) \geq \frac{1}{2}(u(x)+u(y))$, i.e. $Z(x,y) \geq 0$. Thus $u$ is midpoint-concave, which implies by a result of Jensen \cite{jensen1906fonctions} that $u$ is concave.
\end{proof}

\begin{proof}[Proof of Theorem \ref{thm:parabolic}]
Define the parabolic two-point function,
\begin{align*}
    Z(x,y,t) = u(\gamma_{x,y}(0),t)-\frac{1}{2}\left( u(\gamma_{x,y}(-1),t)+ u(\gamma_{x,y}(1),t)\right).
\end{align*}
By the same argument as in the proof of Theorem \ref{thm:elliptic}, the boundary condition \eqref{eqn:parabolic-bcs} implies $Z$ can only attain a minimum at a spatial interior point $(x,y,t_0) \in \Omega \times \Omega\times [0,T]$. If $t_0=0$ then $Z(x,y,t_0)\geq 0$, hence $u(\cdot, 0)$ is concave and we are done. Otherwise $t_0\in (0, T]$, and at $(x,y,t_0)$ we have
\begin{align}
    \partial_t Z \leq 0
\end{align}
and
\begin{align} \label{eq:second-order-5}
    -(\nabla^2 u)_z \leq -\frac{1}{2}((\nabla^2 u)_x + (\nabla^2 u)_y)(V(\cdot),V(\cdot)) \leq -\frac{1}{2}((\nabla^2 u)_x + (\nabla^2 u)_y),
\end{align}
where $V = \mathrm{diag}\left(\frac{1}{2v_\alpha(1)}\right)_{\alpha=1,\ldots,n}$ and \eqref{eq:second-order-5} is justified similarly to \eqref{eq:15091} above. Thus, using the hypotheses on $f$ and $b$, at time $t=t_0$ we get
\begin{align}
    b\left(z, u(z),\left|\nabla u(z) \right|\right) &= \partial_t u(z) +f\left(\left|\nabla u(z)\right|,-\nabla^2 u(z)\right)\\
    &\leq \frac{1}{2}\left(\partial_t u(x) +\partial_t u(y)\right)+\frac{1}{2}f\left(|\nabla u(x)|, -(\nabla^2 u)_x\right) + \frac{1}{2}f\left(|\nabla u(y)|, -(\nabla^2 u)_y\right)  \\ 
    &= \frac{1}{2}b(x,u(x),|\nabla u(x)|) + \frac{1}{2}b(y,u(y),|\nabla u(y)|) \\
    &\leq \frac{1}{2}b(x,u(x),|\nabla u(z)|) + \frac{1}{2}b(y,u(y),|\nabla u(z)|) \\ 
    & \leq b\left(z, \frac{1}{2}\left(u(x)+u(y)\right),\left|\nabla u(z)\right|\right).
\end{align}
Finally, since $b$ is strictly decreasing in the second slot, it follows that $u(z) \geq \frac{1}{2}(u(x)+u(y))$ at $t=t_0$ and thus $Z(x,y,t_0)\geq 0$. This implies that $u(\cdot,t)$ is concave for each $t \in (0,T]$, as desired.
\end{proof}

\section{Boundary conditions and examples} \label{sec:applications}

In this section, we mention some potential applications of Theorems \ref{thm:elliptic} and \ref{thm:parabolic}. To make our results more amenable to applications, we will first relax structural conditions on the PDEs and determine a suitable alternative for the boundary condition \eqref{eq:mat-bdry-condition}. As before, we take $\Omega \subset \mathcal{M}$ to have geodesically convex closure.

\subsection{Relaxing structural assumptions}

In this subsection we take $f(|\nabla u|,-\nabla^2 u) = -\Delta u$ so that the PDEs \eqref{eq:elliptic-pde-2} and \eqref{eq:parabolic-pde} become semilinear elliptic and parabolic equations respectively. We will show that Theorems \ref{thm:elliptic} and \ref{thm:parabolic} still hold when the assumption that $b$ is \emph{strictly} decreasing in the second variable is relaxed to the assumption that it is \emph{non-increasing} in that variable. In the elliptic case this follows from a perturbation lemma in \cite{Korevaar-Convex}, which we recall below.

\begin{lemma} \label{lem:perturbation}
    Suppose $\Omega' \subset\subset \Omega$ has smooth boundary and let $u \in C^2(\Omega)$ be a solution to 
    \begin{align} \label{eq:Mat-PDE-special}
        -\Delta u = b(x,u,|\nabla u|^2)
    \end{align}
    where $b: \bar{\Omega} \times \R \times [0,\infty) \to \R$ is as in Theorem \ref{thm:elliptic}, except that we may allow $b$ to be \emph{non-increasing} in the second variable. Assume also that $b$ is smooth up to the boundary of its domain. Then for small enough $0<\varepsilon<\varepsilon_1$ there exists a solution $v^\varepsilon$ to the the perturbed problem
    \begin{align} \label{eq:perturbed-PDE}
        \begin{cases}
            -\Delta v^\varepsilon = b(x,v^\varepsilon,|\nabla v^\varepsilon|^2) - \varepsilon v^\varepsilon &\text{in } \Omega', \\
            v^\varepsilon = u &\text{on } \del\Omega',
        \end{cases}
    \end{align}
    such that $v^\varepsilon = u+\varepsilon w^\varepsilon$, with $\|w^\varepsilon\|_{C^{2,\alpha}(\Omega')} \leq M$ and $M$ being independent of $\varepsilon$.
\end{lemma}
\begin{proof}
This follows from \cite{Korevaar-Convex}*{Lemma 1.5}, which only makes use of elliptic estimates that hold true in our setting by working in local coordinates.
\end{proof}
By taking an exhaustion of the domain $\Omega$ and letting $\varepsilon \to 0$, Lemma \ref{lem:perturbation} implies the following corollary:
\begin{corollary}\label{cor:elliptic-improve}
Theorem \ref{thm:elliptic} holds for $f(|\nabla u|,-\nabla^2 u) = -\Delta u$ under the weaker assumption that $b$ is non-increasing in the second variable.
\end{corollary}

Similarly, using a trick attributed to Evans in \cite{Korevaar-Convex}, we observe that if $u$ solves
\EQ{
\partial_t u = \Delta u + b(x,u,|\nabla u|^2)
}
with $b(x,u,|\nabla u|^2)$ satisfying $\frac{\partial b}{\partial u}<0$, then 
\begin{align}
    v(x,t) =e^{-\varepsilon t} u(x,t)
\end{align}
solves 
\begin{align}
\partial_t v= \Delta v + \left(-\varepsilon v   + e^{-\varepsilon t} b(x,e^{\varepsilon t}v,e^{\varepsilon t}|\nabla v|)\right) 
&=  \Delta v + \Tilde{b}(t,x,v,|\nabla v|)
\end{align}
which in turn implies $\frac{\partial \Tilde{b}}{\partial u} < 0$ and hence $v$ is concave by Theorem \ref{thm:parabolic}, provided the boundary conditions are met. Sending $\varepsilon\to 0$, we have $v\to u$ uniformly which implies that Theorem \ref{thm:parabolic} holds for $u$ under the assumption that $\frac{\partial b}{\partial u}\leq 0.$ Thus we get

\begin{corollary}\label{cor:parabolic-improve}
Theorem \ref{thm:parabolic} holds for $f(|\nabla u|,-\nabla^2 u) = -\Delta u$ under the weaker assumption that $b$ is non-increasing in the second variable.
\end{corollary}

\subsection{Boundary condition}

In Theorem \ref{thm:elliptic} and \ref{thm:parabolic}, the boundary conditions \eqref{eq:mat-bdry-condition} and \eqref{eqn:parabolic-bcs} ensure that the two-point function $Z$ attains its minimum in the interior of $\bar{\Omega} \times \bar{\Omega}$. However, the condition demands checking a differential inequality for the solution at pairs of points, one lying on the boundary and the other in the interior of the domain, which \emph{a priori} seems difficult to do since it requires knowledge of the unknown function $u$ at interior points of the domain.

One alternative is to assume a growth condition at the boundary of the domain. 
\begin{theorem} \label{thm:u-to-infty}
Assume $u$ is as in Theorem \ref{thm:elliptic} except instead of the boundary condition \eqref{eq:mat-bdry-condition}, we assume that $u \to -\infty$ at the boundary of $\Omega$. Then $u$ is concave. 
\end{theorem}
\begin{proof}
It suffices to show that $Z$ (now defined on $\Omega \times \Omega$) is nonnegative. First argue that $Z$ attains its infimum; this for instance follows from \cite{Kawohl}*{Lemma 3.11}. Then we can argue as in the proof of Theorem \ref{thm:elliptic} to prove that $Z(x,y) \geq 0$, implying that $u$ is concave.
\end{proof}

\subsection{Examples} 
We collect a few examples of PDEs to which our main theorems can be applied. 
\newline
Consider the torsion problem
\begin{align}
    -\Delta u= 1 \quad \text{in }\Omega.
\end{align}
Then Lemma \ref{cor:elliptic-improve} shows that $u$ is concave provided \eqref{eq:mat-bdry-condition} holds. In the parabolic setting, if $u = u(x,t)$ solves the standard heat equation
\begin{align}
    \partial_t u = \Delta u \quad \text{in } \Omega,
\end{align}
then $v=-\log u$ satisfies 
\begin{align}
    \partial_t v = \Delta v - |\nabla v|^2.
\end{align}
Then Lemma \ref{cor:parabolic-improve} implies that $v(\cdot,t)$ is concave, provided that $v(\cdot,0)$ is concave and $v$ satisfies \eqref{eqn:parabolic-bcs}. In other words, the heat flow preserves log-convexity of $u$. On the other hand, using Theorem \ref{thm:u-to-infty} we can establish concavity on $\mathcal{M}$ for solutions of Liouville's equation (see Example 3.3 in the book \cite{Kawohl}), 
\begin{align} \label{eqn:liouville}
    \begin{cases}
        -\Delta u = ce^{-du} & \text{in } \Omega, \\
        u \to -\infty & \text{uniformly as } d(x,\del\Omega) \to 0,
    \end{cases}
\end{align}
for constants $c,d\geq 0$. Indeed, this PDE meets the structural conditions to apply Theorem \ref{thm:u-to-infty} and consequently we get that $u$ is concave.

Another typical way to enforce the constraint $u\to -\infty$ is to apply a transformation $v=-g(u)$ where $u$ is a solution to some elliptic PDE and $g:\R^+\to \R$ is a $C^1$ function satisfying $\lim_{u\to 0^+}g(u)=+\infty$. Theorem \ref{thm:u-to-infty} can then be applied to $v$. For instance, if $u$ is a positive solution of
\begin{align}
    \begin{cases}
        \Delta u = \frac{|\nabla u|^2}{u} + u^p & \text{in } \Omega, \\
        u = 0 & \text{on } \del\Omega,
    \end{cases}
\end{align}
then $v= \log u$ satisfies $v(x) \to -\infty$ as $x\to \partial \Omega.$ Furthermore, $v$ satisfies
\begin{align}
    -\Delta v &= -\left(\frac{\Delta u}{u} - \frac{|\nabla u|^2}{u^2}\right) = -u^{p-1} = -e^{v(1-p)}.
\end{align}
If $p \leq 1$, then $-e^{v(1-p)}$ is non-increasing and concave in $v$, so Theorem \ref{thm:u-to-infty} (together with Corollary \ref{cor:elliptic-improve} if $p=1$) implies that $v$ is a concave function.
\subsection*{Acknowledgments}
The authors would like to thank Tobias Colding and William Minicozzi for their insightful advice and encouragement. We also appreciate the stimulating discussions with Gabriel Khan and Malik Tuerkoen, and we are grateful to the referees for suggesting a generalization of our main result and providing detailed feedback. M.L. acknowledges support from a Croucher Scholarship.
\subsection*{Data Availability Statement}
Data sharing is not applicable to this article as no datasets were generated or analyzed during the current study.
\subsection*{Ethics declarations}
\subsubsection*{Conflict of interest}
All authors declare that they have no conflict of interest.

\bibliography{Refs}
\end{document}